\documentclass[12pt,reqno]{amsart}

\usepackage{amsmath, amsthm, amssymb,amsfonts, commath, mathrsfs, csquotes} \usepackage{xcolor}

\newcommand{\Z}{\mathbb{Z}}

\newcommand{\CO}{\mathcal{O}}
\usepackage[nameinlink,capitalize]{cleveref}
\newtheorem{thm}{Theorem} [section]
\newtheorem{lem}[thm]{Lemma}
\newtheorem{prop}[thm]{Proposition}

\theoremstyle{definition}

\theoremstyle{remark}

\numberwithin{equation}{section}

\begin{document}
	
\title[Linear recurrence and repdigits]{Linear recurrence sequences and palindromic concatenations of two repdigits in base $\beta$}

\author{Ruofan Li} 
\email{liruofan@jnu.edu.cn}
\address{Department of Mathematics, Jinan University, Guangzhou, China, 510632}

\begin{abstract}
Let $\beta$ be a non-unit real algebraic integer greater than one and $\{a_{n}\}_{n \geq 0}$ be a sequence satisfying a linear recurrence relation $a_{n+3}=aa_{n+2}+ba_{n+1}+ca_{n}$. Under certain conditions, we prove that the number of $a_{n}$ which are palindromic concatenations of two repdigits in base $\beta$ is finite.
\end{abstract}

\maketitle

\section{Introduction}
A repdigit in base ten is an integer of the form
\[\overline{\underbrace{d\cdots d}_{m \text{ times }}} = d\sum_{i=0}^{m-1} 10^{i}, \]
where $d \in \{0,1,\ldots,9\}$, and we say an integer is a palindromic concatenation of two repdigits if it equals
\[\overline{\underbrace{d_{1}\cdots d_{1}}_{l \text{ times }} \underbrace{d_{2}\cdots d_{2}}_{m \text{ times }} \underbrace{d_{1}\cdots d_{1}}_{l \text{ times }}}= d_{1}\sum_{i=l+m}^{2l+m-1} 10^{i} + d_{2}\sum_{i=l}^{l+m-1} 10^{i} + d_{1}\sum_{i=0}^{l-1} 10^{i}, \]
for some $d_{1},d_{2} \in \{0,1,\ldots,9\}$ and $l,m \geq 1$. In \cite{CD21}, Chalebgwa and Ddamulira showed that $151$ and $616$ are the only two Padovan numbers which are palindromic concatenations of two distinct repdigits. Recently, similar results were obtained by Ddamulira, Emong and Mirumbe \cite{DEM24} for Narayana's cows sequence, Batte \cite{Batte25} for Lucas number and Ddamulira \cite{Ddamulira25} for Tribonacci-Lucas numbers. 

It is commonly believed that most results in base ten should remain valid in other bases.
Indeed, Ad\'edji, Filipin, Rihane and Togb\'e \cite{AFRT25} proved finiteness of Pell and Pell–Lucas numbers that are concatenations of two repdigits in base $b$ for any integer $b \geq 2$, which generalizes earlier work of Alahmadi, Altassan, Luca and Shoaib \cite{AALS21} in base ten. A similar result of Ray and Bhoi \cite{RB24} determined all Narayana numbers that are concatenations of two repdigits in base $b$ with $2 \leq b \leq 9$.

Let $K$ be a number field with ring of integers $\CO_{K}$ and $\beta \in \CO_{K}$ be a non-unit real algebraic integer greater than one. The main purpose of this article is to prove that for a class of recurrence sequences, each of them only has finitely many terms which are palindromic concatenation of two repdigits in base $\beta$. We say an algebraic integer $\alpha \in \CO_{K}$ has base $\beta$ expansion $\overline{(d_{m}\cdots d_{1})}_{\beta}$ if 
\[\alpha = d_{m} \beta^{m-1} + \cdots+d_{1},\]
where $d_{1}, \ldots, d_{m} \in \{0,1,\dots,\abs{N(\beta)}-1\}$, $d_{m} \neq 0$ and $N(\beta)$ is the norm of $\beta$. Note that unlike the representation of positive integers in integer base $b \geq 2$, an algebraic integer could have no base $\beta$ expansion or more than one base $\beta$ expansions, the interested reader is referred to \cite{Kovacs81}.  

Let $\{a_{n}\}_{n \geq 0}$ be a sequence satisfying a linear recurrence relation $a_{n+3}=aa_{n+2}+ba_{n+1}+ca_{n}$, where $a,b,c$ are algebraic integers. It is well-known that if the characteristic polynomial $f(X)=X^3-aX^2-bX-c$ has three distinct roots $z_1, z_2, z_3$, then 
\begin{equation*} 
a_{n} = A_{1} z_1^n +A_{2} z_2^n  +A_{3} z_3^n \text{ for all } n\geq 0,
\end{equation*} 
where $A_{1},A_{2},A_{3}$ are determined by $a_0$, $a_1$ and $a_2$. The main theorem of this paper is as follows.

\begin{thm} \label{thm:main}
Suppose $K$ is a number field, $\beta>1$ is a real algebraic integer in $K$ with $\abs{N(\beta)}>1$ and $f(x) \in \CO_{K}[X]$ is a polynomial with three distinct roots $z_1, z_2, z_3$, such that $z_1>1$ is a real root, $\abs{z_2} \leq 1$ and $\abs{z_3} \leq 1$. Let $A_{1},A_{2},A_{3}$ be three algebraic numbers. If $A_{1} \neq 0$ and $z_1 \notin K$, then there are only finitely many solutions $$(l,m,n,d_{1},d_{2}) \in \Z_{\geq 1}^{3} \times \{1,\dots,\abs{N(\beta)}-1\} \times \{0,1,\dots,\abs{N(\beta)}-1\}$$ to the Diophantine equation
\begin{equation} \label{eq:thm}
A_{1} z_1^n +A_{2} z_2^n  +A_{3} z_3^n = d_{1}\sum_{i=l+m}^{2l+m-1} \beta^{i} + d_{2}\sum_{i=l}^{l+m-1} \beta^{i} + d_{1}\sum_{i=0}^{l-1} \beta^{i}.
\end{equation}
\end{thm}

In other words, Theorem \ref{thm:main} says that the recurrence sequence $a_{n} = A_{1} z_1^n +A_{2} z_2^n  +A_{3} z_3^n$ only has finitely many terms which are palindromic concatenations of two repdigits in base $\beta$. 

\section{Preliminary} \label{sec:pre}

Throughout this article, we use $x \ll y$ and $y \gg x$ to mean that $\abs{x} \leq C\abs{y}$ for some constant $C>0$ that is independent of $(l,m,n,d_{1},d_{2})$. If $x \ll y$ and $y \ll x$, we write $x = O(y)$. Note that since $d_{1}$ and $d_{2}$ take only finitely many values, it suffices to require $C$ being independent of $l,m,n$.

We begin by showing that $2l+m$ and $n$ are of the same magnitude.
\begin{lem} \label{lem:n_and_2l+m}
If equation \eqref{eq:thm} holds, then
\[n \log z_{1} = (2l+m) \log\beta + O(1).\]
In particular, $2l+m = O(n)$.
\end{lem} 
\begin{proof}
Since $A_{1} \neq 0$, $z_{1}>1$, $\abs{z_2} \leq 1$ and $\abs{z_3} \leq 1$, the equations
\begin{align*}
\abs{a_{n}} &= \abs{A_{1} z_1^n +A_{2} z_2^n  +A_{3} z_3^n} \leq \abs{A_{1}} \abs{z_1}^n +\abs{A_{2}} \abs{z_2}^n  +\abs{A_{3}} \abs{z_3}^n, \\
\abs{a_{n}} &= \abs{A_{1} z_1^n +A_{2} z_2^n  +A_{3} z_3^n} \geq \abs{A_{1}} \abs{z_1}^n -\abs{A_{2}} \abs{z_2}^n  -\abs{A_{3}} \abs{z_3}^n
\end{align*}
lead to $a_{n} = O(z_{1}^{n}).$ 

Since $\beta>1$, we have
\begin{align*}
a_{n} &\leq \left(\abs{N(\beta)}-1\right)\sum_{i=0}^{2l+m-1} \beta^{i} \\
&= \left(\abs{N(\beta)}-1\right) \frac{\beta^{2l+m}-1}{\beta-1} \\
&< \beta^{2l+m} \frac{\abs{N(\beta)}-1}{\beta-1}.
\end{align*}
Hence $z_{1}^{n} \ll \beta^{2l+m}$ and thus $n \log z_{1} \leq (2l+m) \log\beta + O(1)$.

Then we combine $a_{n} = O(z_{1}^{n})$ with the fact that 
\begin{align*}
a_{n} &\geq d_{1} \beta^{2l+m-1}+d_{1} > \beta^{2l+m-1}
\end{align*}
to deduce $z_{1}^{n} \gg \beta^{2l+m}$, and thus $n \log z_{1} \geq (2l+m) \log\beta + O(1)$.
\end{proof}

For any nonzero algebraic number $\eta$, let $h(\eta)$ denote its logarithmic Weil height and we set $h(0)=0$. Basic properties of height function can be found in many books on Diophantine geometry; see, for instance, \cite{BG06}. The ones that we will use are recorded below.

\begin{prop} \label{prop:height}
For any nonzero algebraic numbers $\eta_{1},\eta_{2}$ and any integer $s$, we have
\begin{align*}
h(\eta_{1} \pm \eta_{2}) &\leq h(\eta_{1})+h(\eta_{2})+\log 2,\\
h(\eta_{1} \eta_{2}) &\leq h(\eta_{1}) + h(\eta_{2}), \\
h(\eta_{1}^{s}) &= \abs{s} h(\eta_{1}).
\end{align*}
\end{prop}

The following theorem on linear forms in logarithms is a consequence of the main result of Matveev \cite{Matveev00}.

\begin{thm}\cite[Theorem 9.4.]{BMS06} \label{thm:linear-forms-log}
Suppose $x_{1},\ldots,x_{t}$ are $t$ nonzero elements in a number field $L$ with degree $D$ and $b_{1},\ldots,b_{t}$ be rational integers. Set
\begin{align*}
B &= \max\{\abs{b_{1}},\ldots,\abs{b_{t}}\}, \\
\Lambda &= x_{1}^{b_{1}} \cdots x_{t}^{b_{t}}-1, \\
h'(x_{j}) &= \max\{Dh(x_{j}), \abs{\log{x_{j}}},0.16\}.
\end{align*}
If $\Lambda \neq 0$, then
\[\log\abs{\Lambda} > -3 \cdot 30^{t+4} (t+1)^{5.5} D^{2} (1+\log D)(1+\log tB) \prod_{1 \leq j \leq t} h'(x_{j}).\]
\end{thm}

\section{Key Lemmas} \label{sec:lemma}
First we show that the powers of $z_{1}$ cannot be in $K$, which will be important when we apply Theorem \ref{thm:linear-forms-log}.

\begin{lem} \label{lem:n-power}
If $z_{1} \notin K$, then for any $n \geq 1$, we have $z_{1}^{n} \notin K$.
\end{lem}
\begin{proof}
Assume $z_{1}^{n} \in K$. Let $\overline{K(z_{1})}$ be the Galois closure of $K(z_{1})$ over $K$ and choose $\sigma \in \mathrm{Gal}(\overline{K(z_{1})}/K)$ that does not fix $z_{1}$. Since $z_{1}$ is a root of $f(X)$, we must have $\sigma(z_{1}) = z_{2}$ or $z_{3}$, hence $\abs{\sigma(z_{1})} \leq 1$. On the other hand, $z_{1}^{n} \in K$ implies that $\sigma(z_{1}^{n}) = z_{1}^{n}$, which leads to the following contradiction.
\[1 \geq \abs{\sigma(z_{1})}^{n} = \abs{\sigma(z_{1}^{n})} = \abs{z_{1}^{n}} = \abs{z_{1}}^{n}>1.\]
\end{proof}

Suppose $(l,m,n,d_{1},d_{2})$ is a solution to equation \eqref{eq:thm}, then
\begin{align*} 
&A_{1}z_1^n +A_{2}z_2^n  +A_{3}z_3^n \\
=& d_{1} \beta^{l+m}  \frac{\beta^{l}-1}{\beta-1} + d_{2} \beta^{l} \frac{\beta^{m}-1}{\beta-1} + d_{1} \frac{\beta^{l}-1}{\beta-1} \nonumber\\
=& \frac{1}{\beta-1} \left(d_{1}\beta^{2l+m} - (d_{1}-d_{2})\beta^{l+m} + (d_{1}-d_{2})\beta^{l} - d_{1}\right).
\end{align*}
So
\begin{align} \label{eq:main}
&(\beta-1)(A_{1}z_1^n +A_{2}z_2^n  +A_{3}z_3^n) \nonumber\\
=& d_{1}\beta^{2l+m} - (d_{1}-d_{2})\beta^{l+m} + (d_{1}-d_{2})\beta^{l} - d_{1}.
\end{align}

We are going to rewrite equation \eqref{eq:main} in three different ways and apply Theorem \ref{thm:linear-forms-log} to deduce three inequalities.

\begin{lem} \label{lem:lamdba1}
If equation \eqref{eq:thm} holds, then
\[l \ll \log n.\]
\end{lem}
\begin{proof}
By equation \eqref{eq:main}, we have 
\begin{align*}
&\abs{(\beta-1)A_{1}z_1^n - d_{1}\beta^{2l+m}} \\
=&\abs{-(\beta-1)(A_{2}z_2^n  +A_{3}z_3^n) - (d_{1}-d_{2})\beta^{l+m} + (d_{1}-d_{2})\beta^{l} - d_{1} } \\
\leq& \abs{(\beta-1)(A_{2}z_2^n  +A_{3}z_3^n)} + \abs{ (d_{1}-d_{2})}\beta^{l+m} + \abs{(d_{1}-d_{2})}\beta^{l} + \abs{d_{1}} \\
\leq&O(1)+\left(\abs{N(\beta)}-1\right)(\beta^{l+m}+\beta^{l}+1) \\
\ll& \beta^{l+m}.
\end{align*}
Dividing both sides by $d_{1}\beta^{2l+m}$ and note that $d_{1}$ is bounded, we obtain
\begin{equation} \label{eq:lambda1}
\abs{\frac{(\beta-1)A_{1}}{d_{1}}z_1^n \beta^{-2l-m} - 1} \ll \beta^{-l}.
\end{equation}
Let $\Lambda_{1}=\Lambda_{1}(n,l,m) = \frac{(\beta-1)A_{1}}{d_{1}}z_1^n \beta^{-2l-m} - 1$, we are going to show that $\Lambda_{1} =0$ for at most one $n$. Assume $$ \Lambda_{1}(n,l,m)=\Lambda_{1}(n',l',m')=0,$$ then 
\[\frac{(\beta-1)A_{1}}{d_{1}}z_1^n \beta^{-2l-m} = \frac{(\beta-1)A_{1}}{d_{1}}z_1^{n'} \beta^{-2l'-m'},\]
hence
\[z_1^{n-n'} = \beta^{2l+m-2l'-m'} \in K. \]
Therefore Lemma \ref{lem:n-power} implies that $n=n'$.

For $n$ big enough such that $\Lambda_{1} \neq 0$, apply Theorem \ref{thm:linear-forms-log} with
\[x_{1}=\frac{(\beta-1)A_{1}}{d_{1}},\, x_{2} = z_{1},\, x_{3}=\beta, \, b_{1}=1,\, b_{2}=n,\, b_{3}=-2l-m.\]
Note that $x_{1},x_{2},x_{3}$ are independent of $l,m,n$ and $B=\max\{1,n,2l+m\} \ll n$ by Lemma \ref{lem:n_and_2l+m}, so we have 
\begin{equation*}
\log\abs{\Lambda_{1}} \gg -(1+\log 3n).
\end{equation*}
Combine this with equation \eqref{eq:lambda1}, we obtain 
\[-(1+\log 3n) \ll \log (\beta^{-l}),\] thus $l \ll \log n.$
\end{proof}

\begin{lem} \label{lem:lamdba2}
If equation \eqref{eq:thm} holds, then
\[m \ll l\log n.\]
\end{lem}
\begin{proof}
By equation \eqref{eq:main}, we have 
\begin{align*}
&\abs{(\beta-1)A_{1}z_1^n - d_{1}\beta^{2l+m}+ (d_{1}-d_{2})\beta^{l+m}} \\
=&\abs{-(\beta-1)(A_{2}z_2^n +A_{3}z_3^n) + (d_{1}-d_{2})\beta^{l} - d_{1} } \\
\leq& \abs{-(\beta-1)(A_{2}z_2^n +A_{3}z_3^n)} + \abs{(d_{1}-d_{2})}\beta^{l} + \abs{d_{1}} \\
\leq&O(1)+\left(\abs{N(\beta)}-1\right)(\beta^{l}+1) \\
\ll& \beta^{l}.
\end{align*}
Dividing both sides by $(d_{1}\beta^{l}-d_{1}+d_{2})\beta^{l+m}$ and note that $d_{1}$ and $d_{2}$ are bounded, we obtain
\begin{equation} \label{eq:lambda2}
\abs{\frac{(\beta-1)A_{1}}{d_{1}\beta^{l}-d_{1}+d_{2}}z_1^n \beta^{-l-m} - 1} \ll \beta^{-m}.
\end{equation}
Let $\Lambda_{2}=\Lambda_{2}(n,l,m) = \frac{(\beta-1)A_{1}}{d_{1}\beta^{l}-d_{1}+d_{2}}z_1^n \beta^{-l-m} - 1$, we are going to show that $\Lambda_{2} =0$ for at most one $n$. Assume $$ \Lambda_{2}(n,l,m)=\Lambda_{2}(n',l',m')=0,$$ then 
\[\frac{(\beta-1)A_{1}}{d_{1}\beta^{l}-d_{1}+d_{2}}z_1^n \beta^{-l-m}= \frac{(\beta-1)A_{1}}{d_{1}\beta^{l'}-d_{1}+d_{2}}z_1^{n'} \beta^{-l'-m'},\]
hence
\[z_1^{n-n'} = \beta^{l+m-l'-m'}\frac{d_{1}\beta^{l}-d_{1}+d_{2}}{d_{1}\beta^{l'}-d_{1}+d_{2}} \in K. \]
Therefore Lemma \ref{lem:n-power} implies that $n=n'$

For $n$ big enough such that $\Lambda_{2} \neq 0$, apply Theorem \ref{thm:linear-forms-log} with
\[x_{1}=\frac{(\beta-1)A_{1}}{d_{1}\beta^{l}-d_{1}+d_{2}},\, x_{2} = z_{1},\, x_{3}=\beta, \, b_{1}=1,\, b_{2}=n,\, b_{3}=-l-m.\]
For $x_{1}$, we have
\begin{align*}
h(x_{1}) &\leq h((\beta-1)A_{1}) + h(d_{1}\beta^{l}-d_{1}+d_{2}) \\
&\leq h((\beta-1)A_{1}) + h(d_{1}\beta^{l}) + h(d_{1}-d_{2}) + \log 2 \\
&\leq h((\beta-1)A_{1}) + h(d_{1})+lh(\beta) + h(d_{1}-d_{2}) + 2\log 2 \\
&\leq lh(\beta) + O(1)
\end{align*}
and
\begin{align*}
\abs{\log{x_{1}}} &\leq \abs{\log ((\beta-1)A_{1})} + \abs{\log (d_{1}\beta^{l}-d_{1}+d_{2})} \\
&\leq O(1) + \log (d_{1}\beta^{l}) + \abs{\log \left(1-\frac{d_{1}-d_{2}}{d_{1}\beta^{l}}\right)}.
\end{align*}
If $d_{1} \geq d_{2}$, then
\begin{align*}
\abs{\log \left(1-\frac{d_{1}-d_{2}}{d_{1}\beta^{l}}\right)} &= -\log \left(1-\frac{d_{1}-d_{2}}{d_{1}}\beta^{-l}\right) \\
&\leq -\log \left(1-\beta^{-l}\right) \\
&\leq -\log \left(1-\beta^{-1}\right).
\end{align*}
If $d_{1} < d_{2}$, then
\begin{align*}
\abs{\log \left(1-\frac{d_{1}-d_{2}}{d_{1}\beta^{l}}\right)} &= \log \left(1+\frac{d_{2}-d_{1}}{d_{1}}\beta^{-l}\right) \\
&\leq \log \left(1+\left(\abs{N(\beta)}-2\right) \beta^{-l}\right) \\
&\leq \log \left(1+\left(\abs{N(\beta)}-2\right) \beta^{-1}\right)
\end{align*}
So in both cases, we have
\begin{equation*}
\abs{\log{x_{1}}} \leq \log (d_{1}\beta^{l})+O(1) = l \log \beta + O(1).
\end{equation*}
Therefore
\[h'(x_{1}) = \max\{Dh(x_{1}), \abs{\log{x_{1}}},0.16\} \ll l.\]

Note that $x_{2},x_{3}$ are independent of $l,m,n$, and $B=\max\{1,n,l+m\} \ll n$ by Lemma \ref{lem:n_and_2l+m}, so we have 
\begin{equation*}
\log\abs{\Lambda_{2}} \gg -(1+\log 3n) l.
\end{equation*}
Combine this with equation \eqref{eq:lambda2}, we obtain 
\[\log (\beta^{-m}) \gg -(1+\log 3n) l,\]
thus $m \ll l\log n.$
\end{proof}

\begin{lem} \label{lem:lamdba3}
For $n$ big enough, we have
\[n \ll (l+2m)\log n.\]
\end{lem}
\begin{proof}
By equation \eqref{eq:main}, we have 
\begin{align*}
&\abs{(\beta-1)A_{1}z_1^n - d_{1}\beta^{2l+m}+ (d_{1}-d_{2})\beta^{l+m} - (d_{1}-d_{2})\beta^{l}} \\
=&\abs{-(\beta-1)(A_{2}z_2^n +A_{3}z_3^n) - d_{1} } \\
\leq& \abs{-(\beta-1)(A_{2}z_2^n +A_{3}z_3^n)} + \abs{d_{1}} \\
=&O(1).
\end{align*}
Dividing both sides by $(\beta-1)A_{1}z_1^n$, we obtain
\begin{equation} \label{eq:lambda3}
\abs{1 - \frac{d_{1}\beta^{l+m}- (d_{1}-d_{2})\beta^{m} + (d_{1}-d_{2})}{(\beta-1)A_{1}} \beta^{l}z_1^{-n}} \ll A_{1}^{-1}z_1^{-n}.
\end{equation}
Let $$\Lambda_{3}=\Lambda_{3}(n,l,m) = \frac{d_{1}\beta^{l+m} - (d_{1}-d_{2})\beta^{m} + (d_{1}-d_{2})}{(\beta-1)A_{1}} \beta^{l}z_1^{-n} - 1,$$ we are going to show that $\Lambda_{3} =0$ for at most one $n$. Assume $$ \Lambda_{3}(n,l,m)=\Lambda_{3}(n',l',m')=0,$$ then 
\begin{align*}
&\frac{d_{1}\beta^{l+m}- (d_{1}-d_{2})\beta^{m} + (d_{1}-d_{2})}{(\beta-1)A_{1}} \beta^{l}z_1^{-n} \\
=&\frac{d_{1}\beta^{l'+m'}- (d_{1}-d_{2})\beta^{m'} + (d_{1}-d_{2})}{(\beta-1)A_{1}} \beta^{l'}z_1^{-n'},
\end{align*}
hence
\[z_1^{n-n'} = \beta^{l-l'}\frac{d_{1}\beta^{l+m}- (d_{1}-d_{2})\beta^{m} + (d_{1}-d_{2})}{d_{1}\beta^{l'+m'}- (d_{1}-d_{2})\beta^{m'} + (d_{1}-d_{2})} \in K. \]
Therefore Lemma \ref{lem:n-power} implies that $n=n'$

For $n$ big enough such that $\Lambda_{3} \neq 0$, apply Theorem \ref{thm:linear-forms-log} with
\[x_{1}=\frac{d_{1}\beta^{l+m}- (d_{1}-d_{2})\beta^{m} + (d_{1}-d_{2})}{(\beta-1)A_{1}},\, x_{2} = \beta,\, x_{3}=z_{1}, \]
\[b_{1}=1,\, b_{2}=l,\, b_{3}=-n.\]
For $x_{1}$, we have
\begin{align*}
h(x_{1}) &\leq h((\beta-1)A_{1}) + h(d_{1}\beta^{l+m}- (d_{1}-d_{2})\beta^{m} + (d_{1}-d_{2})) \\
&\leq h((\beta-1)A_{1}) + h(d_{1}\beta^{l+m}) + h((d_{1}-d_{2})\beta^{m}) + h(d_{1}-d_{2}) + 2\log 2 \\
&\leq h((\beta-1)A_{1}) + h(d_{1})+(l+m)h(\beta) + mh(\beta) + 2h(d_{1}-d_{2}) + 4\log 2 \\
&\leq (l+2m)h(\beta) + O(1)
\end{align*}
and
\begin{align*}
\abs{\log{x_{1}}} &\leq \abs{\log ((\beta-1)A_{1})} + \abs{\log (d_{1}\beta^{l+m}- (d_{1}-d_{2})\beta^{m} + (d_{1}-d_{2}))} \\
&\leq O(1) + \log (d_{1}\beta^{l+m}) + \abs{\log \left(1-\frac{(d_{1}-d_{2})(\beta^{m}-1)}{d_{1}\beta^{l+m}}\right)}.
\end{align*}
When $d_{1} \geq d_{2}$, we have
\begin{align*}
\abs{\log \left(1-\frac{(d_{1}-d_{2})(\beta^{m}-1)}{d_{1}\beta^{l+m}}\right)}&=-\log \left(1-\frac{d_{1}-d_{2}}{d_{1}} \frac{\beta^{m}-1}{\beta^{l+m}} \right) \\
&\leq -\log \left(1-\frac{\beta^{m}-1}{\beta^{l+m}}\right) \\
&= -\log \left(1-\frac{1-\beta^{-m}}{\beta^{l}}\right) \\
&< -\log \left(1-\frac{1}{\beta^{l}}\right) \\
&\leq -\log \left(1-\frac{1}{\beta}\right).
\end{align*}
When $d_{1} < d_{2}$, we have
\begin{align*}
\abs{\log \left(1-\frac{(d_{1}-d_{2})(\beta^{m}-1)}{d_{1}\beta^{l+m}}\right)}&= \log \left(1+\frac{d_{2}-d_{1}}{d_{1}} \frac{\beta^{m}-1}{\beta^{l+m}} \right) \\
&\leq \log \left(1+\left(\abs{N(\beta)}-2\right) \frac{\beta^{m}-1}{\beta^{l+m}} \right) \\
&= \log \left(1+\left(\abs{N(\beta)}-2\right) \frac{1-\beta^{-m}}{\beta^{l}} \right) \\
&< \log \left(1+\left(\abs{N(\beta)}-2\right) \frac{1}{\beta^{l}} \right) \\
&\leq \log \left(1+\left(\abs{N(\beta)}-2\right) \frac{1}{\beta} \right).
\end{align*}
So in both cases, we have
\begin{align*}
\abs{\log{x_{1}}} \leq \log (d_{1}\beta^{l+m})+O(1) = (l+m) \log \beta + O(1).
\end{align*}
Note that $x_{2},x_{3}$ are independent of $l,m,n$, and
$B=\max\{1,l,n\} \ll n$ by Lemma \ref{lem:n_and_2l+m}, so we have 
\begin{align*}
\log\abs{\Lambda_{3}} &\gg -(1+\log 3n) \max\{Dh(x_{1}), \abs{\log{x_{1}}},0.16\} \\
&\gg -(1+\log 3n) (l+2m).
\end{align*}
Combine this with equation \eqref{eq:lambda3}, we obtain 
\[\log (z_{1}^{-n})+O(1) \gg -(1+\log 3n) (l+2m),\]
thus $n \ll (l+2m)\log n.$
\end{proof}

\section{Proof of Theorem \ref{thm:main} and further discussion} \label{sec:further}

Now we combine Lemma \ref{eq:lambda1}, Lemma \ref{eq:lambda2} and Lemma \ref{eq:lambda3} to show that
\begin{align*}
n &\ll (l+2m)\log n &\text{ by Lemma \ref{eq:lambda3},} \\
&\ll (l+2l\log n)\log n &\text{ by Lemma \ref{eq:lambda2},} \\
&\ll (\log n+2(\log n)^{2})\log n &\text{ by Lemma \ref{eq:lambda1},} \\
&\ll (\log n)^{3}.
\end{align*}

There are only finitely many $n$ satisfying $n \ll (\log n)^{3}$. For each $n$, the possible values of $l$ and $m$ are bounded by Lemma \ref{lem:n_and_2l+m}. Therefore equation \eqref{eq:thm} only has finitely many solutions. This finishes the proof of Theorem \ref{thm:main}. \qed

We remark that Theorem \ref{thm:main} is not effective. Despite that most inequalities in the proof can be made explicit, we only know that $\Lambda_{1}, \Lambda_{2}, \Lambda_{3}$ are nonzero for $n$ big enough without any explicit bound. In order to have an effective result, certain information about $A_{1}$ needs to be known. For instance, if $A_{1} \in K$, then we could show that $\Lambda_{1}, \Lambda_{2}, \Lambda_{3}$ are nonzero for all $n \geq 1$, hence obtain an effective version of Theorem \ref{thm:main}, although the bound might be too big to carry out computation in practice.

In the present article, we assume that $\beta$ is real. One naturally wonders whether similar conclusion holds for complex $\beta$. The base $\beta$ expansion can be defined in the same way, and the theorem on linear forms in logarithms remains valid. We believe that for most recurrence sequences, there are only finitely many terms which are palindromic concatenations of two repdigits in complex base $\beta$, since such numbers are very rare. However, our method does not work, as when $\beta$ is complex, 
\[\abs{d_{1}\sum_{i=l+m}^{2l+m-1} \beta^{i} + d_{2}\sum_{i=l}^{l+m-1} \beta^{i} + d_{1}\sum_{i=0}^{l-1} \beta^{i}}\]
could be small even when $l$ and $m$ are large, so $2l+m$ and $n$ may not be of the same magnitude, that is, Lemma \ref{lem:n_and_2l+m} fails. 

Finally we briefly discuss a generalization of Theorem \ref{thm:main}. Since the conditions that $z_{2}$ and $z_{3}$ are roots of $f(x)$ with absolute values less than or equal to one are only used to deduce $z_{1}^{n} \notin K$ for all $n \geq 1$ and $A_{2}z_2^n  +A_{3}z_3^n = O(1)$, Theorem \ref{thm:main} could be extended to the following form.

\begin{thm} \label{thm:general}
Suppose $K$ is a number field, $\beta>1$ is a real algebraic integer in $K$ with $\abs{N(\beta)}>1$. Let $A$ be a nonzero algebraic number and $\{B_{n}\}_{n \geq 1}$ be a bounded sequence of algebraic numbers. If $z>1$ is a real number such that $z^{n} \notin K$ for all $n \geq 1$, then there are only finitely many solutions $$(l,m,n,d_{1},d_{2}) \in \Z_{\geq 1}^{3} \times \{1,\dots,\abs{N(\beta)}-1\} \times \{0,1,\dots,\abs{N(\beta)}-1\}$$ to the Diophantine equation
\begin{equation} \label{eq:general}
A z^n + B_{n} = d_{1}\sum_{i=l+m}^{2l+m-1} \beta^{i} + d_{2}\sum_{i=l}^{l+m-1} \beta^{i} + d_{1}\sum_{i=0}^{l-1} \beta^{i}.
\end{equation}
\end{thm}

In particular, let $\{a_{n}\}_{n\geq 0}$ be an unbounded recurrence sequence whose characteristic polynomial has distinct roots where one real root $z>1$ satisfying $z \notin K$ and the absolute values of all other roots are less than or equal to one. A similar argument as Lemma \ref{lem:n-power} could show that $z^{n} \notin K$ for all $n \geq 1$. Therefore Theorem \ref{thm:general} implies that there are only finitely many terms which are palindromic concatenations of two repdigits in base $\beta$.

{\noindent \bf  Acknowledgements}. The author thank the referees for careful reading of the manuscript and many helpful suggestions. This research was supported by National Natural Science Foundation of China grant number 12401006.


\begin{thebibliography}{10}

\bibitem{AFRT25}
K.~N. Ad\'edji, A.~Filipin, S.~E. Rihane, and A.~Togb\'e.
\newblock Pell or {P}ell-{L}ucas numbers as concatenations of two repdigits in
  base {$b$}.
\newblock {\em Rev. R. Acad. Cienc. Exactas F\'is. Nat. Ser. A Mat. RACSAM},
  119(1):Paper No. 15, 16, 2025.

\bibitem{AALS21}
A.~Alahmadi, A.~Altassan, F.~Luca, and H.~Shoaib.
\newblock Fibonacci numbers which are concatenations of two repdigits.
\newblock {\em Quaest. Math.}, 44(2):281--290, 2021.

\bibitem{Batte25}
H.~Batte.
\newblock Lucas numbers that are palindromic concatenations of two distinct
  repdigits.
\newblock {\em Math. Pannon. (N. S.)}, 31(1):22--33, 2025.

\bibitem{BG06}
E.~Bombieri and W.~Gubler.
\newblock {\em Heights in {D}iophantine geometry}, volume~4 of {\em New
  Mathematical Monographs}.
\newblock Cambridge University Press, Cambridge, 2006.

\bibitem{BMS06}
Y.~Bugeaud, M.~Mignotte, and S.~Siksek.
\newblock Classical and modular approaches to exponential {D}iophantine
  equations. {I}. {F}ibonacci and {L}ucas perfect powers.
\newblock {\em Ann. of Math. (2)}, 163(3):969--1018, 2006.

\bibitem{CD21}
T.~P. Chalebgwa and M.~Ddamulira.
\newblock Padovan numbers which are palindromic concatenations of two distinct
  repdigits.
\newblock {\em Rev. R. Acad. Cienc. Exactas F\'is. Nat. Ser. A Mat. RACSAM},
  115(3):Paper No. 108, 14, 2021.

\bibitem{Ddamulira25}
M.~Ddamulira.
\newblock Tribonacci-{L}ucas {N}umbers that are {P}alindromic {C}oncatenations
  of {T}wo {D}istinct {R}epdigits.
\newblock {\em Math. Pannon. (N. S.)}, 31(2):196--208, 2025.

\bibitem{DEM24}
M.~Ddamulira, P.~Emong, and G.~I. Mirumbe.
\newblock Palindromic concatenations of two distinct repdigits in {N}arayana's
  cows sequence.
\newblock {\em Bull. Iranian Math. Soc.}, 50(3):Paper No. 35, 16, 2024.

\bibitem{Kovacs81}
B.~Kov\'acs.
\newblock Canonical number systems in algebraic number fields.
\newblock {\em Acta Math. Acad. Sci. Hungar.}, 37(4):405--407, 1981.

\bibitem{Matveev00}
E.~M. Matveev.
\newblock An explicit lower bound for a homogeneous rational linear form in
  logarithms of algebraic numbers. {II}.
\newblock {\em Izv. Ross. Akad. Nauk Ser. Mat.}, 64(6):125--180, 2000.

\bibitem{RB24}
P.~K. Ray and K.~Bhoi.
\newblock Narayana numbers which are concatenations of two base {$b$}
  repdigits.
\newblock {\em Integers}, 24:Paper No. A20, 10, 2024.

\end{thebibliography}
\end{document}